\documentclass[11pt]{article}
\usepackage{amsmath}
\usepackage{amstext,amsbsy}
\usepackage{epsfig,multicol}
\usepackage{amsthm}
\usepackage{amssymb,latexsym}
\theoremstyle{plain}
\newtheorem{theorem}{Theorem}[section]
\newtheorem{lemma}[theorem]{Lemma}

\newtheorem{proposition}[theorem]{Proposition}
\newtheorem{example}[theorem]{Example}

\theoremstyle{definition}
\newtheorem{definition}[theorem]{Definition}

\newtheorem*{remark}{Remark}

\numberwithin{equation}{section}

\begin{document}

\title{A Bijection Between Partially Directed Paths in the Symmetric Wedge  and Matchings}
\author{Svetlana~Poznanovi\'{c}}
\date{}

\maketitle

\begin{abstract}
We give a bijection between partially directed paths in the
symmetric wedge $y= \pm x$ and matchings, which sends north steps
to nestings. This gives a bijective proof of a result of Prellberg
et al. that was first discovered through the corresponding
generating functions: the number of partially directed paths
starting at the origin confined to the symmetric wedge $y= \pm x$
with $k$ north steps is equal to the number of matchings on $[2n]$
with $k$ nestings.
\end{abstract}



\section{Introduction}

The purpose of this paper is to give a bijective proof of a fact
that was discovered unexpectedly and connects two seemingly
different branches of combinatorics. One of the branches is the
study of matchings and set partitions and, more specifically, the
statistics crossings and nestings. The other one is the study of
directed paths in the plane.

Based on Touchard's work~\cite{Touchard}, Riordan~\cite{Riordan}
derived a formula for the number of matchings with $k$ crossings.
Since then, a lot of results connected to this topic have been
obtained. We mention a few. M.~de Sainte-Catherine
in~\cite{Sainte-Catherine} bijectively shows that the number of
matchings with $k$ crossings is equal to the number of matchings
with $k$ nestings. This bijection also implies symmetric joint
distribution of crossings and nestings. More than two decades
later, Kasraoui and Zeng, in~\cite{Kasraoui}, extended this
bijection to show that the same result holds for set partitions.
Martin Klazar~\cite{Klazar06} studied the distribution of these
statistics on subtrees of the generating tree of matchings, and
the same questions for set partitions were studied in~\cite{PYan}.

In another line of work, Prellberg et al. in~\cite{Rpr} worked on
founding a generating function of self-avoiding partially directed
paths in the wedge $y=\pm px$ consisting of east, north and south
steps. Using the kernel method they were able to derive explicitly
the generating function for the case $p=1$. The generating
function revealed that the number of such paths which end at
$(n,-n)$ with $k$ north steps is the same as the number of
matchings on $[2n]$ with $k$ nestings. For a nice survey of the
history of the problem and how this fact was discovered
see~\cite{Rubey}.

A matching on the set $[2n]=\{1,\dots,2n\}$ is a family of $n$
two-element disjoint subsets of $[2n]$. In particular, it is a
set-partition with all the blocks of size two. It is convenient to
represent a matching with its standard diagram consisting of arcs
connecting $2n$ vertices on a horizontal line (see Figure
\ref{fig:matching}). The vertices are numbered in increasing order
from left to right. The set of all matchings of $[2n]$ is denoted
by $\mathcal{M}_n$. We say that two edges $(a,b)$ and $(c,d)$ form
a crossing if $a<c<b<d$ (i.e. if the cross) and they form a
nesting if $a<c<d<b$ (i.e. if one covers the other). If they are
neither crossed nor nested we say they form an alignment. The
number of nestings in a matching $M$ is denoted by $ne(M)$.
\\
\\

\begin{figure}[h]
\begin{center}
\includegraphics[width=8cm]{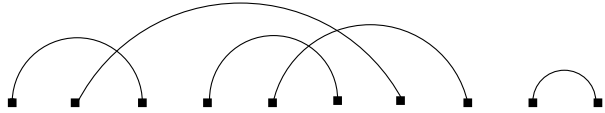}
\end{center}
\caption{\emph{Diagram of a matching with 10 vertices and edges:
$e_1=(1,3),e_2=(2,7),e_3=(4,6),e_4=(5,8)$, and $e_5=(9,10)$. This
matching has 3 crossings formed by the pairs of edges:
$(e_1,e_2),(e_2,e_4),$ and $(e_3,e_4)$, one nesting  $(e_2,e_3)$,
and all the other pairs of edges form alignments.}}
\label{fig:matching}
\end{figure}

A partially directed path in the plane is a path starting at the
origin and consisting of unit east, north, and south steps. We
consider all such paths confined to the symmetric wedge defined by
the lines $y=\pm x$. Let $\mathcal{P}_n$ be the set of all such
paths  ending at the line $y=-x$ with $n$ horizontal steps.

\begin{theorem}
There is a bijection $\Phi:\mathcal{P}_n \rightarrow
\mathcal{M}_n$ that takes the number of north steps of $P \in
\mathcal{P}_n$ to the number of nestings of $\Phi(P)$.
\end{theorem}

\begin{remark} While preparing the present paper, we found out about the very recent work of Martin
Rubey~\cite{Rubey} in which he presents a bijective proof of the
same result. However, our bijection is different from Rubey's, as
illustrated in Example \ref{T:phi}. In particular, $\Phi$ may be
of special interest in the study of matchings because a key part
of it is a bijection on matchings which, unlike the other
bijections used in the literature, does not preserve the type of
the matching, i.e., the sets of minimal and maximal elements of
the blocks. This may give further insight into the interaction
between matchings of different type when various statistics of
matchings are studied.
\end{remark}

\section{Definition and properties of the bijection $\Phi$}

Below we define a bijection $\Phi:\mathcal{P}_n \rightarrow
\mathcal{M}_n$ that takes the number of north steps of $P \in
\mathcal{P}_n$ to the number of nestings of $\Phi(P)$. The map
$\Phi$ is defined as the composition of two maps: $\Phi=\phi \circ
\psi$, where $\psi:\mathcal{P}_n \rightarrow \mathcal{M}_n$ and
$\phi:\mathcal{M}_n \rightarrow \mathcal{M}_n$.

\subsection{Bijection  $\psi$ from $\mathcal{P}_n$ to $\mathcal{M}_n$}

Every path $P \in \mathcal{P}_n$ is determined by the
$y$-coordinates of its east steps, i.e., a sequence $a_1, \dots,
a_n$ of integers such that $ -(i-1) \leq a_i \leq i-1$. Set
$b_i=a_{n+1-i}+n+1-i$. Note that $ 1 \leq b_i \leq 2(n+1-i)-1$.
Define a matching $M$ on $[2n]$  by connecting the first available
vertex from the left to the $b_i$-th available vertex to its
right, one by one for each $i=1,\dots, n$ in that order. Note that
before the $i$-th step there are $2(n+1-i)$ vertices that are not
connected yet, so each step is possible. We define $\psi(P)=M$. It
is not hard to see that knowing M, one can reverse the steps one
by one and find the $b_i$'s, which determine a path $P$. So $\psi$
is a bijection. Figure \ref{fig:mappsi} shows a path $P \in
\mathcal{P}_7$ and $\psi(P)$.

\begin{definition} Let $M \in \mathcal{M}_n$.
Suppose the edges $e_1,\dots,e_n$ of $M$ are ordered according to
their left endpoints in ascending order. Suppose $e_i=(a,b)$ and
$e_{i+1}=(c,d)$. Define
\[st_i(M):=
\begin{cases}
\left| \{v: d \leq v \leq b, v \text{ is a vertex of } e_k, k> i
\} \right|, &\text{if $e_i$
and $e_{i+1}$ are nested}\\
0,  &\text{otherwise}
\end{cases}
\]
and
\[st(M)=\sum_{i=1}^{n-1}st_i(M).\]
\end{definition}

\begin{figure}[h]
\begin{picture}(100,-120)
\put(90,-120){\includegraphics[height=4cm,width=17cm]{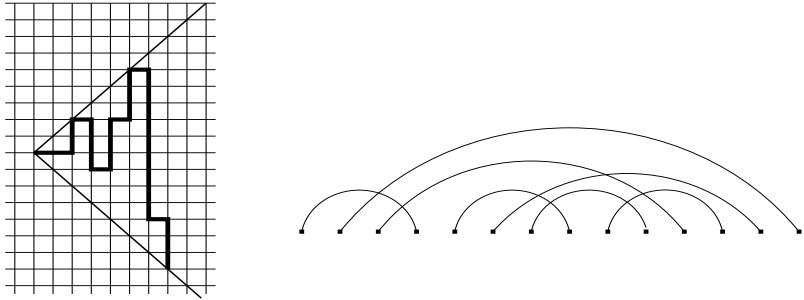}}
\end{picture}
\vspace{4cm} \caption{\emph{Path $P \in \mathcal{P}_7$ and the
corresponding matching $\psi(P).$}}
\label{fig:mappsi}
\end{figure}
\smallskip

\begin{lemma} The number of north steps of $P$ is equal to
$st(\psi(P))$.
\end{lemma}

\begin{proof}

Let $M=\psi(P)$. The number of north steps of $P$ is
\begin{eqnarray}
\sum_{a_{i+1} > a_i}{(a_{i+1}-a_i)} =\sum_{b_{n-i} \geq
b_{n-i+1}+2}{(b_{n-i}-b_{n-i+1}-1)}
\end{eqnarray}
So, it suffices to show that
\begin{eqnarray}
st_i(M)=
\begin{cases}
b_i-b_{i+1}-1, \; &\text{if $b_{i} \geq b_{i+1}+2$} \\
0, &\text{otherwise}
\end{cases}
\end{eqnarray}
 After the $i$-th edge $e_i$ is
drawn in the construction of $M$, there are $b_i-1$ unconnected
vertices below it. In the case $b_{i} \geq b_{i+1}+2$, we have
$b_i-1 \geq b_{i+1}+1$ which implies $e_{i+1}$ is nested below
$e_i$ and $st_i(M)=b_i-b_{i+1}-1$. In the other case, when $b_{i}
< b_{i+1}+2$, we have $b_i-1 < b_{i+1}+1$ and hence the edge
$e_{i+1}$ and $e_i$ are crossed (if $b_i>1$) or aligned (if
$b_i=1$). In either case, $st_i(M)=0$.
\end{proof}

\subsection{Bijection  $\phi$ from $\mathcal{M}_n$ to $\mathcal{M}_n$}
 We describe $\phi$ by a series of transformations on the diagrams of the matchings.
 This map preserves the first edge. For $M \in
\mathcal{M}_n$, $N=\phi(M)$ is constructed inductively as follows.
If $n=1$ set $\phi(M)=M$. If $n>1$, let $M_1$ be the matching
obtained from $M$ by deleting its first edge $e_1=(1,r)$ and let
$N_1=\phi(M_1)$. Let $N_2$ be the matching obtained by adding back
the edge $e_1$ in the same position as it was in $M$. Denote by
$e_2$ the second edge of $N_2$ (which was also the second edge of
$M$). There are three cases:
\begin{itemize}
\item [\emph{\textbf{case 1:}}] $e_1$ and $e_2$ were aligned

In this case set $N=\phi(M)=N_2$.

\item [\emph{\textbf{case 2:}}] $e_1$ and $e_2$ were crossed

Let $f_2=e_2 =(l_2,r_2), f_3=(l_3,r_3), \dots, f_k=(l_k,r_k)$ be
the edges in $N_2$ crossing $e_1$ ordered by their left endpoints
$2=l_2<l_3<\cdots <l_k$. Rearrange them in the following way:
connect $r_2$ to $l_3$, $r_3$ to $l_4, \dots, r_{k-1}$ to $l_k$.
Finally, insert one additional vertex right before $r$ and connect
it to $r_k$. Delete the vertex $l_2$ and renumber the remaining
vertices (see Figure~\ref{fig:case2}). Note that the position of
the first edge in the matching $N$ obtained this way is the same
as in $M$. Set $\phi(M)=N$.

\begin{figure}[h]
\begin{center}
\includegraphics[height=2cm]{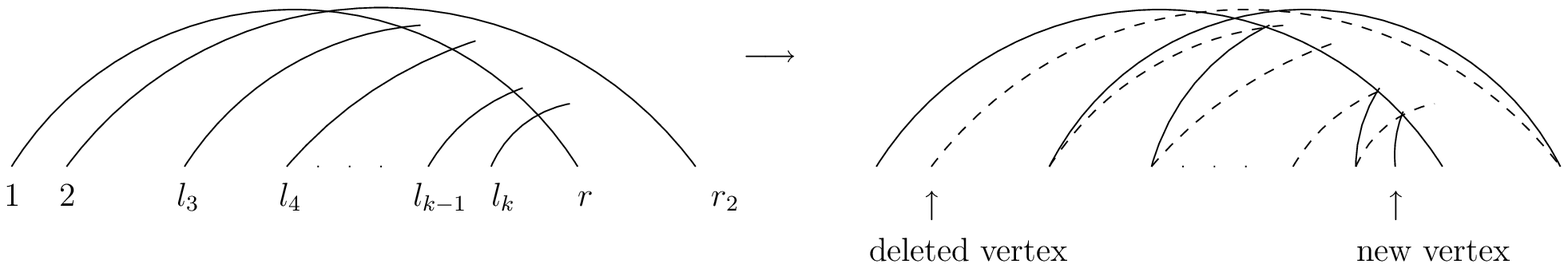}
\end{center}
\caption{\emph{Definition of $\phi$ when $e_1$ and $e_2$ are
crossed. Dashed lines are used to represent edges whose left
endpoints have been changed.}} \label{fig:case2}
\end{figure}

\item [\emph{\textbf{case 3:}}] $e_1$ and $e_2$ were nested

In $N_2$, let $f_1 =(l_1,r_1), \dots, f_p=(l_p,r_p)$ be the edges
crossing both $e_1=(1,r)$ and $e_2=(2,q)$, and let $f_{p+1}
=(l_{p+1},r_{p+1}), \dots, f_{p+s}=(l_{p+s},r_{p+s})$ be the edges
crossing $e_1$ but not $e_2$, such that $l_1<\cdots
<l_{p}<q<l_{p+1}< \cdots <l_{p+s}$. For easier notation let
$\{l_1<\cdots <l_{p}<q<l_{p+1}< \cdots <l_{p+s}\}=\{v_1<\cdots
<v_{p}<v_{p+1}<v_{p+2}< \cdots <v_{p+s+1}\}$. Add one vertex right
before $r$ and connect it to $v_{s+1}$. "Rearrange" the edges
$f_1, \dots, f_{p+s}$ so that $r_1,\dots,r_{p+s}$ are connected to
$v_1,\dots,v_s,v_{s+2},\dots,v_{p+s+1}$ in that order.  Finally,
delete the vertex 2 and renumber the remaining vertices. See
Figure~\ref{fig:case3} for an illustration when $p=3$ and $s=2$.
Call the matching obtained this way $N$. The first edge of $N$ is
the same as in $M$. Set $\phi(M)=N$.
\end{itemize}

\begin{figure}[h]
\begin{center}
\includegraphics[height=3cm]{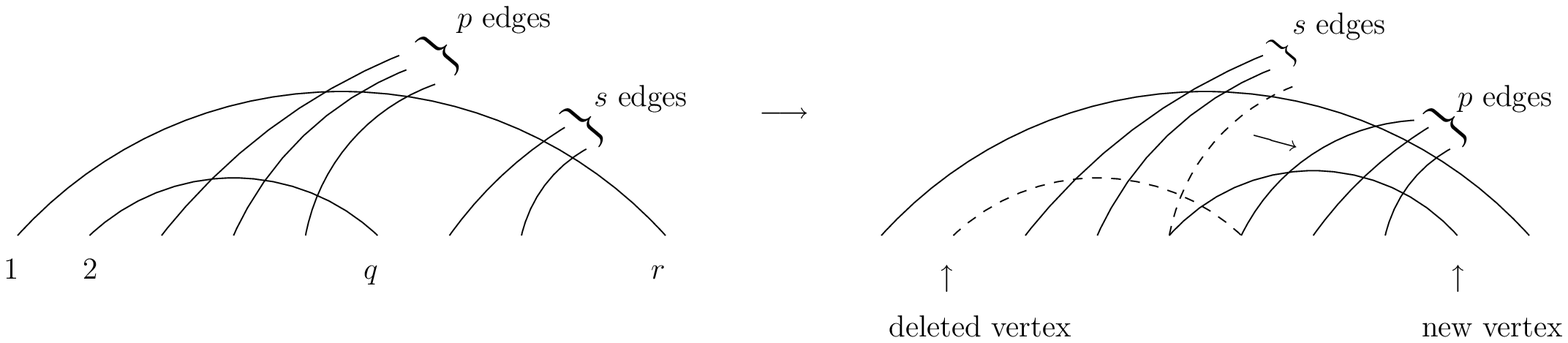}
\end{center}
\caption{\emph{Example of case 3 for $p=3$ and $s=2$. }}
\label{fig:case3}
\end{figure}

\begin{example}\label{T:phi}
Figure~\ref{fig:fi} shows step-by-step construction of $\phi(M)$
for the matching $M$ from Figure~\ref{fig:mappsi}. So, for the
path $P$ given in Figure~\ref{fig:mappsi}, the corresponding
matching is
$\Phi(P)=\{(1,4),(2,14),(3,12),(5,8),(6,9),(7,11),(10,13)\}$. Note
that the image of $P$ under Rubey's bijection defined
in~\cite{Rubey} is
$\{(1,4),(2,14),(3,11),(5,8),(6,9),(7,13),(10,12)\}$. Hence the
two bijections are different.
\end{example}

\begin{figure}[h]
\begin{center}
\includegraphics[width=13cm]{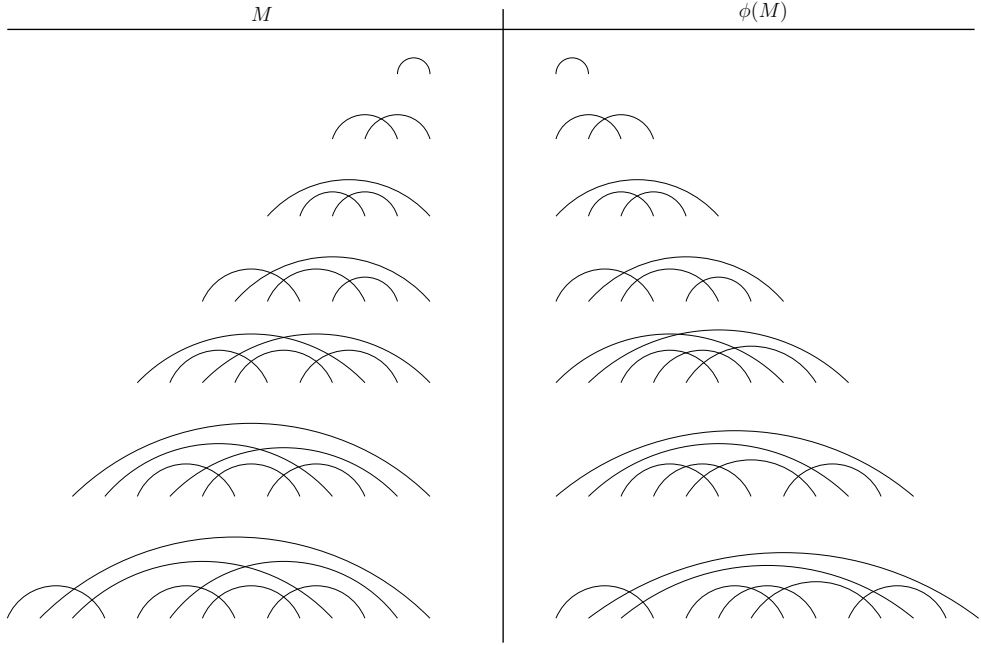}
\end{center}
\caption{\emph{Example of construction of $\phi(M)$}}
\label{fig:fi}
\end{figure}

\begin{theorem}
The map $\phi$ is a bijection and $ne(\phi(M))=st(M)$.
\end{theorem}

\begin{proof} To show that $\phi$ is bijective, we explain how to define the inverse map.  Note that the matching resulting from case 1 above
has the property that its first edge is (1,2). In the matching
resulting from case 2 (case 3 respectively), the vertex preceding
the right endpoint of the first edge $e_1$ is a left endpoint
(right endpoint respectively) of an edge different than $e_1$.
Since all the steps in the definition of $\phi$ are invertible, we
simply perform the inverse steps of the corresponding case.

It is left to prove $ne(\phi(M))=st(M)$. For shortness, for any
matching $M$, let $ne(e,M)$ denote the number of edges in $M$
below the edge $e$. Let $M$, $M_1$, $N_1$, $N_2$, and $N$ be the
same as in the definition of $\phi$. By inductive hypothesis,
$ne(N_1)=st(M_1)=st(M)-st_1(M)$. So we just need to prove
\begin{equation}\label{E:need}
ne(N)=ne(N_1)+st_1(M)
\end{equation}
It is clear that
\begin{equation} \label{E:clear}
ne(N_2)=ne(N_1)+ ne(e_1,N_2)
\end{equation}

In the first case of the definition of $\phi$, ~\eqref{E:need}
clearly follows since $st_1(M)=0$ and we do not add nestings to
$N_1$ by adding back $e_1$.

In the second case, $st_1(M)=0$, so we need to show that
$ne(N)=ne(N_1)$. To this end, if $e$ is an edge in $N_2$ different
from $f_2,\dots, f_k$ (notation from the definition of $\phi$),
let $r(e)$ be the edge in $N$ that corresponds to $e$ in the
obvious way, and let $r(f_i)$ be the edge with right endpoint
$r_i$, for $i=2, \dots, k$.  It is clear that
$ne(e,N_2)=ne(r(e),N)$
 for any edge $e \notin \{e_1,f_2,\dots, f_k\}$. Note that the left endpoint of $r(f_i)$ in $N$ is $l_{i}-1$ because the vertex $2$ from $N_2$ was deleted (see Figure~\ref{fig:case2}).
 So, for $2 \leq i <
 k$
 \begin{align}
  ne(f_i,N_2)&-ne(r(f_i),N)= \notag \\
 &=\left|\{\text{edges in $N$ below $e_1$ with left endpoint between $l_i-1$ and $l_{i+1}-1$}\}\right | \label{E:ne1}\\
  ne(f_k,N_2)&-ne(r(f_k),N)= \notag \\
 &=\left |\{\text{edges in $N$ below $e_1$ with left endpoint between $l_k-1$ and
 $r$}\}\right | \label{E:ne2}
 \end{align}
 By subtracting the following equalities
 \begin{align}
 ne(N_2)&=\sum_{i=2}^k{ne(f_i,N_2)}+\sum_{e \notin \{f_2,\dots,
 f_k\}}{ne(e,N_2)} \label{E:N2} \\
 ne(N)&=\sum_{i=2}^k{ne(r(f_i),N)}+\sum_{e \notin \{f_2,\dots,
 f_k\}}{ne(r(e),N)} \label{E:N}
 \end{align}
 and using~\eqref{E:ne1} and ~\eqref{E:ne2} we get
 \begin{equation}
 ne(N_2)-ne(N)=ne(e_1,N)=ne(e_1,N_2)
 \end{equation}
This together with~\eqref{E:clear} gives $ne(N)=ne(N_1)$.

In the third case, similarly, denote by $r(f_i)$ the edge in $N$
that ends with vertex $r_i$, $i=1,\dots, p+s$, by $r(e_2)$ the
edge that ends with the vertex $r-1$, and for every other edge $e$
in $N_2$, denote by $r(e)$ the edge in $N$ that corresponds to $e$
in the natural way. In $N_2$, define $a$ to be the number of edges
below $e_1$ and crossing $e_2=(2,q)$ and $b$ to be the number of
those edges below $e_1$ with a left endpoint right of $q$.  In
what follows, $v_i$ are the vertices defined in case 3 of the
definition of $\phi$. Then
\begin{align}
st_1(M)&=1+a+2b+s \label{E:1}\\
ne(r(e_2),N)&= \left| \{\text{edges in $N_2$ below $e_1$ with left
endpoint between
$v_{s+1}$ and $r$}\}\right| \label{E:2}\\
ne(N_2)&=ne(N_1)+ne(e_2,N_2)+1+a+b \label{E:first}\\
ne(N_2)&=ne(e_1,N_2)+ne(e_2,N_2)+\sum_{i=1}^{p+s}{ne(f_i,N_2)}+\sum_{e
\notin
\{e_1,e_2,f_1,\dots, f_{p+s}\}}{ne(e,N_2)} \label{E:second}\\
ne(N)&=ne(e_1,N)+ne(r(e_2),N)+
\sum_{i=1}^{p+s}{ne(r(f_i),N)}+\sum_{e \notin \{e_1,e_2,f_1,\dots,
f_{p+s}\}}{ne(r(e),N_2)} \label{E:third}
\end{align}
To complete the proof, we need to distinguish two cases: $ s\geq
p$ and $p>s$. When $ s\geq p$, close inspection of the
"rearrangement" of the edges reveals:
\begin{align}\label{E:3}
&ne(r(f_i),N)-ne(f_i,N_2)= \notag \\
&=
\begin{cases}
1,  & 1 \leq i \leq p\\
1+ \left| \{\text{edges in $N_2$ below $e_1$ with left vertex
between
$v_i$ and $v_{i+1}$}\}\right|, & p < i \leq s \\
0, & s < i \leq p+s
\end{cases}
\end{align}
while when $p>s$, similar equalities hold:
\begin{align}\label{E:4}
& ne(r(f_i))-ne(f_i)= \notag \\
&=
\begin{cases}
1,  & 1 \leq i \leq s\\
-\left| \{\text{edges in $N_2$ below $e_1$ with left vertex
between $v_i$ and $v_{i+1}$}\}\right|, & s < i \leq p\\
0, & p <  i \leq p+s
\end{cases}
\end{align}
Now, we add the equations~\eqref{E:first} and ~\eqref{E:third} and
subtract~\eqref{E:second} from them. Using \eqref{E:1},
\eqref{E:2}, and \eqref{E:3},i.e.,~\eqref{E:4}, we
get~\eqref{E:need}.
\end{proof}

\subsection{Some properties of $\Phi$}

First we need few definitions. We say that $\{l,l+1,\dots,k\}$ is
a component of a matching $M \in \mathcal{M}_n$ if the
restrictions of $M$ on each of the sets $\{1,\dots,l-1\}$,
$\{l,l+1, \dots,k \}$, and $\{ k+1 ,\dots, n\}$ are matchings
themselves. A matching is called irreducible if it has only one
component. In terms of diagrams, a matching is irreducible if it
cannot be split by vertical bars into disjoint matchings.

A component of a path $P \in \mathcal{P}_n$ is a subsequence of
consecutive steps beginning at $(l,-l)$ and ending at $(k,-k)$
such that both parts of $P$ between $(l,-l)$ and $(k,-k)$, and
between $(k,-k)$ and $(n,-n)$ when translated by the appropriate
vector to the origin represent paths in $\mathcal{P}_{k-l}$ and
$\mathcal{P}_{n-k}$ respectively. A component which does not have
nontrivial subcomponents is called irreducible.

\begin{proposition} For $P \in \mathcal{P}_n$ the following are
true:
\begin{itemize}
\item[(a)] $P$ has $k$ south steps on the line $x=n$ if and only
in $\Phi(P)$, $1$ is connected to $k+1$ .

\item[(b)] The  irreducible components of $P$ read backwards are
in one-to-one correspondence with the irreducible components of
$\Phi(P)$ from left to right.
\end{itemize}
\end{proposition}

\begin{proof}
\begin{itemize}
\item[(a)] From the definition of $\psi$, it is clear that $P$ has
$k$ south steps on the line $x=n$ if and only in $\psi(P)$, $1$ is
connected to $k+1$. Thus, the claim follows from the fact that
$\phi$ preserves the first edge.

\item[(b)] This statement is clearly true if we replace $\Phi$ by
$\psi$. Hence, it suffices to observe that if the irreducible
components of $\psi(P)$ are $C_1, \dots, C_k$, then $\phi(C_1),
\dots, \phi(C_k)$ are the irreducible components of $\Phi(P)$.
\end{itemize}
\end{proof}

\begin{proposition}
If $P$ is a path with no north steps (Dyck path) then $M=\Phi(P)$
is the unique matching with no nestings such that $i$ is a left
endpoint in $M$ exactly when the $(2n+1-i)$-th step of $P$ is a
south step.

\noindent In other words, the set of left and right endpoints  of
$M$ is determined by $P$ traced backwards.
\end{proposition}

\begin{proof}
It follows from the definition of $\psi$ that  the statement is
true for $\psi(P)$. Moreover, since $\psi(P)$ has no nestings,
$\phi$ leaves $\psi(P)$ unchanged.
\end{proof}

\section*{Acknowledgment}

The author would like to thank Catherine Yan for suggesting this
problem, which was posed by A.~Rechnitzer at the 19-th FPSAC in
July, 2007.

\textsc{Svetlana Poznanovi\'{c}: Department of Mathematics, Texas
A\&M
University, College Station, TX 77843 USA} \\
\emph{E-mail address}: \verb"spoznan@math.tamu.edu"

%\bibliographystyle{plain}
%\bibliography{tau}

\end{document}